\documentclass{amsart}
\usepackage{amsmath}
\usepackage{amsfonts}

\newtheorem{theorem}{Theorem}
\theoremstyle{plain}

\newtheorem{corollary}{Corollary}

\newtheorem{lemma}{Lemma}

\newtheorem{proposition}{Proposition}
\newtheorem{remark}{Remark}

\numberwithin{equation}{section}
\input{tcilatex}

\begin{document}
\title[Integral inequalities]{A new generalization of some integral
inequalities and their applications}
\author{\.{I}mdat \.{I}\c{s}can}
\address{Department of Mathematics, Faculty of Arts and Sciences,\\
Giresun University, 28100, Giresun, Turkey.}
\email{imdat.iscan@giresun.edu.tr}
\date{June 10, 2012}

\begin{abstract}
In this paper, a new identity for convex functions is derived. A consequence
of the identity is that we can derive new estimates for the remainder term
of the midpoint, trapezoid, and Simpson formulae for functions whose
derivatives in absolute value at certain power are convex. Some applications
to special means of real numbers are also given.
\end{abstract}

\maketitle

\section{Introduction}

Let $f:I\subseteq \mathbb{R\rightarrow R}$ be a convex function defined on
the interval $I$ of real numbers and $a,b\in I$ with $a<b$. The following
inequality%
\begin{equation}
f\left( \frac{a+b}{2}\right) \leq \frac{1}{b-a}\dint\limits_{a}^{b}f(x)dx%
\leq \frac{f(a)+f(b)}{2}\text{.}  \label{1-1}
\end{equation}

holds. This double inequality is known in the literature as Hermite-Hadamard
integral inequality for convex functions. See \cite{ADK11,DP00,K04,THH12},
the results of the generalization, improvement and extention of the famous
integral inequality (\ref{1-1}).

The following inequality is well known in the literature as Simpson's
inequality .

Let $f:\left[ a,b\right] \mathbb{\rightarrow R}$ be a four times
continuously differentiable mapping on $\left( a,b\right) $ and $\left\Vert
f^{(4)}\right\Vert _{\infty }=\underset{x\in \left( a,b\right) }{\sup }%
\left\vert f^{(4)}(x)\right\vert <\infty .$ Then the following inequality
holds:%
\begin{equation*}
\left\vert \frac{1}{3}\left[ \frac{f(a)+f(b)}{2}+2f\left( \frac{a+b}{2}%
\right) \right] -\frac{1}{b-a}\dint\limits_{a}^{b}f(x)dx\right\vert \leq 
\frac{1}{2880}\left\Vert f^{(4)}\right\Vert _{\infty }\left( b-a\right) ^{2}.
\end{equation*}

\bigskip In recent years many authors have studied error estimations for
Simpson's inequality; for refinements, counterparts, generalizations and new
Simpson's type inequalities, see \cite{ADD09,SA11,SSO10,SSO10a}

In \cite{SSO10a}, Sarikaya et al. obtained inequalities for differentiable
convex mapping which are connected Simpson's inequality, and they used the
following lemma to prove this.

\begin{lemma}
Let $f:I\subset \mathbb{R\rightarrow R}$ be an absolutely continuous mapping
on $I^{\circ }$ such that $f^{\prime }\in L[a,b]$, where $a,b\in I^{\circ }$
with $a<b$. Then the following equality holds:%
\begin{eqnarray*}
&&\frac{1}{6}\left[ f(a)+4f\left( \frac{a+b}{2}\right) +f(b)\right] -\frac{1%
}{b-a}\dint\limits_{a}^{b}f(x)dx \\
&=&\frac{b-a}{2}\dint\limits_{0}^{1}\left[ \left( \frac{t}{2}-\frac{1}{3}%
\right) f^{\prime }\left( \frac{1+t}{2}b+\frac{1-t}{2}a\right) +\left( \frac{%
1}{3}-\frac{t}{2}\right) f^{\prime }\left( \frac{1+t}{2}a+\frac{1-t}{2}%
b\right) \right] dt.
\end{eqnarray*}
\end{lemma}

The main inequality in \cite{SSO10a}, pointed out, is as follows.

\begin{theorem}
Let $f:I\subset \mathbb{R\rightarrow R}$ be a differentiable mapping on $%
I^{\circ }$, such that $f^{\prime }\in L[a,b]$, where $a,b\in I^{\circ }$
with $a<b$. If \ If $\left\vert f^{\prime }\right\vert ^{q}$ is convex on $%
[a,b],\ $ $q>1,$ then the following inequality holds,%
\begin{eqnarray}
&&\left\vert \frac{1}{6}\left[ f(a)+4f\left( \frac{a+b}{2}\right) +f(b)%
\right] -\frac{1}{b-a}\dint\limits_{a}^{b}f(x)dx\right\vert  \label{1-2a} \\
&\leq &\frac{b-a}{12}\left( \frac{1+2^{p+1}}{3\left( p+1\right) }\right) ^{%
\frac{1}{p}}\left\{ \left( \frac{3\left\vert f^{\prime }(b)\right\vert
^{q}+\left\vert f^{\prime }\left( a\right) \right\vert ^{q}}{4}\right) ^{%
\frac{1}{q}}+\left( \frac{3\left\vert f^{\prime }(a)\right\vert
^{q}+\left\vert f^{\prime }\left( b\right) \right\vert ^{q}}{4}\right) ^{%
\frac{1}{q}}\right\} ,  \notag
\end{eqnarray}%
where $\frac{1}{p}+\frac{1}{q}=1.$
\end{theorem}

In \cite{SSO10}, Sarikaya et al. obtained a new upper bound for the
right-hand side of Simpson's inequality for convex mapping:

\begin{corollary}
Let $f:I\subset \left[ 0,\infty \right) \mathbb{\rightarrow R}$ be a
differentiable mapping on $I^{\circ }$, such that $f^{\prime }\in L[a,b]$,
where $a,b\in I^{\circ }$ with $a<b$. If \ If $\left\vert f^{\prime
}\right\vert ^{q}$ is convex on $[a,b],\ $ $q>1,$ then the following
inequality holds,%
\begin{equation}
\left\vert \frac{1}{6}\left[ f(a)+4f\left( \frac{a+b}{2}\right) +f(b)\right]
-\frac{1}{b-a}\dint\limits_{a}^{b}f(x)dx\right\vert  \label{1-2}
\end{equation}%
\begin{equation*}
\leq \frac{b-a}{12}\left( \frac{1+2^{p+1}}{3\left( p+1\right) }\right) ^{%
\frac{1}{p}}\left\{ \left( \frac{\left\vert f^{\prime }(b)\right\vert
^{q}+\left\vert f^{\prime }\left( \frac{a+b}{2}\right) \right\vert ^{q}}{2}%
\right) ^{\frac{1}{q}}+\left( \frac{\left\vert f^{\prime }(a)\right\vert
^{q}+\left\vert f^{\prime }\left( \frac{a+b}{2}\right) \right\vert ^{q}}{2}%
\right) ^{\frac{1}{q}}\right\} ,
\end{equation*}%
where $\frac{1}{p}+\frac{1}{q}=1.$
\end{corollary}

In \cite{K04}, some inequalities of Hermite-Hadamard type for differentiable
convex mappings were presented as follows.

\begin{theorem}
Let $f:I\subset \mathbb{R\rightarrow R}$ be a differentiable mapping on $%
I^{\circ }$, $a,b\in I^{\circ }$ with $a<b$. If \ If $\left\vert f^{\prime
}\right\vert $ is convex on $[a,b],$ then the following inequality holds,%
\begin{equation}
\left\vert f\left( \frac{a+b}{2}\right) -\frac{1}{b-a}\dint%
\limits_{a}^{b}f(x)dx\right\vert \leq \frac{b-a}{4}\left( \frac{\left\vert
f^{\prime }(a)\right\vert +\left\vert f^{\prime }(b)\right\vert }{2}\right) .
\label{1-3}
\end{equation}
\end{theorem}

In this paper, in order to provide a unified approach to establish midpoint
inequality, trapezoid inequality and Simpson's inequality for functions
whose derivatives in absolute value at certain power are convex, we derive a
general integral identity for convex functions. Finally some applications
for special means of real numbers are provided.

\section{Main results}

In order to prove our main theorems, we need the following Lemma.

\begin{lemma}
\label{2.1}Let $f:I\subset \mathbb{R\rightarrow R}$ be a differentiable
mapping on $I^{\circ }$ such that $f^{\prime }\in L[a,b]$, where $a,b\in I$
with $a<b$ and $\alpha ,\lambda \in \left[ 0,1\right] $. Then the following
equality holds:%
\begin{eqnarray}
&&\lambda \left( \alpha f(a)+\left( 1-\alpha \right) f(b)\right) +\left(
1-\lambda \right) f(\alpha a+\left( 1-\alpha \right) b)-\frac{1}{b-a}%
\dint\limits_{a}^{b}f(x)dx  \label{2-1} \\
&=&\left( b-a\right) \left[ \dint\limits_{0}^{1-\alpha }\left( t-\alpha
\lambda \right) f^{\prime }\left( tb+(1-t)a\right) dt\right.  \notag \\
&&\left. +\dint\limits_{1-\alpha }^{1}\left( t-1+\lambda \left( 1-\alpha
\right) \right) f^{\prime }\left( tb+(1-t)a\right) dt\right] .  \notag
\end{eqnarray}
\end{lemma}

\begin{proof}
We note that%
\begin{equation*}
I=\dint\limits_{0}^{1-\alpha }\left( t-\alpha \lambda \right) f^{\prime
}\left( tb+(1-t)a\right) dt+\dint\limits_{1-\alpha }^{1}\left( t-1+\lambda
\left( 1-\alpha \right) \right) f^{\prime }\left( tb+(1-t)a\right) dt
\end{equation*}%
integrating by parts, we get%
\begin{eqnarray*}
I &=&\left. \left( t-\alpha \lambda \right) \frac{f\left( tb+(1-t)a\right) }{%
b-a}\right\vert _{0}^{1-\alpha }-\dint\limits_{0}^{1-\alpha }\frac{f\left(
tb+(1-t)a\right) }{b-a}dt \\
&&+\left. \left( t-1+\lambda \left( 1-\alpha \right) \right) \frac{f\left(
tb+(1-t)a\right) }{b-a}\right\vert _{1-\alpha }^{1}-\dint\limits_{1-\alpha
}^{1}\frac{f\left( tb+(1-t)a\right) }{b-a}dt \\
&=&\left( 1-\alpha -\alpha \lambda \right) \frac{f\left( \left( 1-\alpha
\right) b+\alpha a\right) }{b-a}+\frac{\alpha \lambda f(a)}{b-a}+\frac{%
\left( 1-\alpha \right) \lambda f(b)}{b-a} \\
&&-\left( -\alpha +\lambda \left( 1-\alpha \right) \right) \frac{f\left(
\left( 1-\alpha \right) b+\alpha a\right) }{b-a}-\dint\limits_{0}^{1}\frac{%
f\left( tb+(1-t)a\right) }{b-a}dt.
\end{eqnarray*}%
Setting $x=tb+(1-t)a,$ and $dx=\left( b-a\right) dt,$ we obtain%
\begin{equation*}
\left( b-a\right) I=\lambda \left( \alpha f(a)+\left( 1-\alpha \right)
f(b)\right) +\left( 1-\lambda \right) f(\alpha a+\left( 1-\alpha \right) b)-%
\frac{1}{b-a}\dint\limits_{a}^{b}f(x)dx
\end{equation*}%
which gives the desired representation (\ref{2-1}).
\end{proof}

\begin{theorem}
\label{2.2}Let $f:I\subset \mathbb{R\rightarrow R}$ be a differentiable
mapping on $I^{\circ }$ such that $f^{\prime }\in L[a,b]$, where $a,b\in
I^{\circ }$ with $a<b$ and $\alpha ,\lambda \in \left[ 0,1\right] $. If $%
\left\vert f^{\prime }\right\vert ^{q}$ is convex on $[a,b]$, $q\geq 1,$
then the following inequality holds:%
\begin{eqnarray}
&&\left\vert \lambda \left( \alpha f(a)+\left( 1-\alpha \right) f(b)\right)
+\left( 1-\lambda \right) f(\alpha a+\left( 1-\alpha \right) b)-\frac{1}{b-a}%
\dint\limits_{a}^{b}f(x)dx\right\vert  \label{2-2} \\
&\leq &\left\{ 
\begin{array}{cc}
\begin{array}{c}
\left( b-a\right) \left\{ \gamma _{2}^{1-\frac{1}{q}}\left( \mu
_{1}\left\vert f^{\prime }(b)\right\vert ^{q}+\mu _{2}\left\vert f^{\prime
}(a)\right\vert ^{q}\right) ^{\frac{1}{q}}\right. \\ 
\left. +\upsilon _{2}^{1-\frac{1}{q}}\left( \eta _{3}\left\vert f^{\prime
}(b)\right\vert ^{q}+\eta _{4}\left\vert f^{\prime }(a)\right\vert
^{q}\right) ^{\frac{1}{q}}\right\} ,%
\end{array}
& \alpha \lambda \leq 1-\alpha \leq 1-\lambda \left( 1-\alpha \right) \\ 
\begin{array}{c}
\left( b-a\right) \left\{ \gamma _{2}^{1-\frac{1}{q}}\left( \mu
_{1}\left\vert f^{\prime }(b)\right\vert ^{q}+\mu _{2}\left\vert f^{\prime
}(a)\right\vert ^{q}\right) ^{\frac{1}{q}}\right. \\ 
\left. +\upsilon _{1}^{1-\frac{1}{q}}\left( \eta _{1}\left\vert f^{\prime
}(b)\right\vert ^{q}+\eta _{2}\left\vert f^{\prime }(a)\right\vert
^{q}\right) ^{\frac{1}{q}}\right\} ,%
\end{array}
& \alpha \lambda \leq 1-\lambda \left( 1-\alpha \right) \leq 1-\alpha \\ 
\begin{array}{c}
\left( b-a\right) \left\{ \gamma _{1}^{1-\frac{1}{q}}\left( \mu
_{3}\left\vert f^{\prime }(b)\right\vert ^{q}+\mu _{4}\left\vert f^{\prime
}(a)\right\vert ^{q}\right) ^{\frac{1}{q}}\right. \\ 
\left. +\upsilon _{2}^{1-\frac{1}{q}}\left( \eta _{3}\left\vert f^{\prime
}(b)\right\vert ^{q}+\eta _{4}\left\vert f^{\prime }(a)\right\vert
^{q}\right) ^{\frac{1}{q}}\right\} ,%
\end{array}
& 1-\alpha \leq \alpha \lambda \leq 1-\lambda \left( 1-\alpha \right)%
\end{array}%
\right.  \notag
\end{eqnarray}%
where 
\begin{equation}
\gamma _{1}=\left( 1-\alpha \right) \left[ \alpha \lambda -\frac{\left(
1-\alpha \right) }{2}\right] ,\ \gamma _{2}=\left( \alpha \lambda \right)
^{2}-\gamma _{1}\ ,  \label{2-2a}
\end{equation}%
\begin{eqnarray}
\upsilon _{1} &=&\frac{1-\left( 1-\alpha \right) ^{2}}{2}-\alpha \left[
1-\lambda \left( 1-\alpha \right) \right] ,  \label{2-2b} \\
\upsilon _{2} &=&\frac{1+\left( 1-\alpha \right) ^{2}}{2}-\left( \lambda
+1\right) \left( 1-\alpha \right) \left[ 1-\lambda \left( 1-\alpha \right) %
\right] ,  \notag
\end{eqnarray}%
\begin{eqnarray}
\mu _{1} &=&\frac{\left( \alpha \lambda \right) ^{3}+\left( 1-\alpha \right)
^{3}}{3}-\alpha \lambda \frac{\left( 1-\alpha \right) ^{2}}{2},\ 
\label{2-2c} \\
\mu _{2} &=&\frac{1+\alpha ^{3}+\left( 1-\alpha \lambda \right) ^{3}}{3}-%
\frac{\left( 1-\alpha \lambda \right) }{2}\left( 1+\alpha ^{2}\right) , 
\notag \\
\mu _{3} &=&\alpha \lambda \frac{\left( 1-\alpha \right) ^{2}}{2}-\frac{%
\left( 1-\alpha \right) ^{3}}{3},  \notag \\
\mu _{4} &=&\frac{\left( \alpha \lambda -1\right) \left( 1-\alpha
^{2}\right) }{2}+\frac{1-\alpha ^{3}}{3},  \notag
\end{eqnarray}%
\begin{eqnarray}
\eta _{1} &=&\frac{1-\left( 1-\alpha \right) ^{3}}{3}-\frac{\left[ 1-\lambda
\left( 1-\alpha \right) \right] }{2}\alpha \left( 2-\alpha \right) ,\ 
\label{2-2e} \\
\eta _{2} &=&\frac{\lambda \left( 1-\alpha \right) \alpha ^{2}}{2}-\frac{%
\alpha ^{3}}{3},  \notag
\end{eqnarray}%
\begin{eqnarray*}
\eta _{3} &=&\frac{\left[ 1-\lambda \left( 1-\alpha \right) \right] ^{3}}{3}-%
\frac{\left[ 1-\lambda \left( 1-\alpha \right) \right] }{2}\left( 1+\left(
1-\alpha \right) ^{2}\right) +\frac{1+\left( 1-\alpha \right) ^{3}}{3}, \\
\eta _{4} &=&\frac{\left[ \lambda \left( 1-\alpha \right) \right] ^{3}}{3}-%
\frac{\lambda \left( 1-\alpha \right) \alpha ^{2}}{2}+\frac{\alpha ^{3}}{3}.
\end{eqnarray*}
\end{theorem}

\begin{proof}
Suppose that $q\geq 1.$ From Lemma \ref{2.1} and using the well known power
mean inequality, we have%
\begin{eqnarray*}
&&\left\vert \lambda \left( \alpha f(a)+\left( 1-\alpha \right) f(b)\right)
+\left( 1-\lambda \right) f(\alpha a+\left( 1-\alpha \right) b)-\frac{1}{b-a}%
\dint\limits_{a}^{b}f(x)dx\right\vert \\
&\leq &\left( b-a\right) \left[ \dint\limits_{0}^{1-\alpha }\left\vert
t-\alpha \lambda \right\vert \left\vert f^{\prime }\left( tb+(1-t)a\right)
\right\vert dt+\dint\limits_{1-\alpha }^{1}\left\vert t-1+\lambda \left(
1-\alpha \right) \right\vert \left\vert f^{\prime }\left( tb+(1-t)a\right)
\right\vert dt\right] \\
&\leq &\left( b-a\right) \left\{ \left( \dint\limits_{0}^{1-\alpha
}\left\vert t-\alpha \lambda \right\vert dt\right) ^{1-\frac{1}{q}}\left(
\dint\limits_{0}^{1-\alpha }\left\vert t-\alpha \lambda \right\vert
\left\vert f^{\prime }\left( tb+(1-t)a\right) \right\vert ^{q}dt\right) ^{%
\frac{1}{q}}\right.
\end{eqnarray*}%
\begin{equation}
\left. +\left( \dint\limits_{1-\alpha }^{1}\left\vert t-1+\lambda \left(
1-\alpha \right) \right\vert dt\right) ^{1-\frac{1}{q}}\left(
\dint\limits_{1-\alpha }^{1}\left\vert t-1+\lambda \left( 1-\alpha \right)
\right\vert \left\vert f^{\prime }\left( tb+(1-t)a\right) \right\vert
^{q}dt\right) ^{\frac{1}{q}}\right\} .  \label{2-3}
\end{equation}

Since $\left\vert f^{\prime }\right\vert ^{q}$ is convex on $[a,b],$ we know
that for $t\in \left[ 0,1\right] $%
\begin{equation*}
\left\vert f^{\prime }\left( tb+(1-t)a\right) \right\vert ^{q}\leq
t\left\vert f^{\prime }(b)\right\vert ^{q}+(1-t)\left\vert f^{\prime
}(a)\right\vert ^{q},
\end{equation*}%
hence, by simple computation%
\begin{equation}
\dint\limits_{0}^{1-\alpha }\left\vert t-\alpha \lambda \right\vert
dt=\left\{ 
\begin{array}{cc}
\gamma _{2}, & \alpha \lambda \leq 1-\alpha \\ 
\gamma _{1}, & \alpha \lambda \geq 1-\alpha%
\end{array}%
\right. ,  \label{2-4}
\end{equation}%
\begin{equation*}
\gamma _{1}=\left( 1-\alpha \right) \left[ \alpha \lambda -\frac{\left(
1-\alpha \right) }{2}\right] ,\ \gamma _{2}=\left( \alpha \lambda \right)
^{2}-\gamma _{1}\ ,
\end{equation*}%
\begin{equation}
\dint\limits_{1-\alpha }^{1}\left\vert t-1+\lambda \left( 1-\alpha \right)
\right\vert dt=\left\{ 
\begin{array}{cc}
\upsilon _{1}, & 1-\lambda \left( 1-\alpha \right) \leq 1-\alpha \\ 
\upsilon _{2}, & 1-\lambda \left( 1-\alpha \right) \geq 1-\alpha%
\end{array}%
\right. ,  \label{2-5}
\end{equation}%
\begin{eqnarray*}
\upsilon _{1} &=&\frac{1-\left( 1-\alpha \right) ^{2}}{2}-\alpha \left[
1-\lambda \left( 1-\alpha \right) \right] , \\
\upsilon _{2} &=&\frac{1+\left( 1-\alpha \right) ^{2}}{2}-\left( \lambda
+1\right) \left( 1-\alpha \right) \left[ 1-\lambda \left( 1-\alpha \right) %
\right] ,
\end{eqnarray*}%
\begin{eqnarray}
\dint\limits_{0}^{1-\alpha }\left\vert t-\alpha \lambda \right\vert
\left\vert f^{\prime }\left( tb+(1-t)a\right) \right\vert ^{q}dt &\leq
&\dint\limits_{0}^{1-\alpha }\left\vert t-\alpha \lambda \right\vert \left[
t\left\vert f^{\prime }(b)\right\vert ^{q}+(1-t)\left\vert f^{\prime
}(a)\right\vert ^{q}\right] dt  \notag \\
&=&\left\{ 
\begin{array}{cc}
\mu _{1}\left\vert f^{\prime }(b)\right\vert ^{q}+\mu _{2}\left\vert
f^{\prime }(a)\right\vert ^{q}, & \alpha \lambda \leq 1-\alpha \\ 
\mu _{3}\left\vert f^{\prime }(b)\right\vert ^{q}+\mu _{4}\left\vert
f^{\prime }(a)\right\vert ^{q}, & \alpha \lambda \geq 1-\alpha%
\end{array}%
\right. ,  \label{2-6}
\end{eqnarray}%
\begin{eqnarray*}
\mu _{1} &=&\frac{\left( \alpha \lambda \right) ^{3}+\left( 1-\alpha \right)
^{3}}{3}-\alpha \lambda \frac{\left( 1-\alpha \right) ^{2}}{2},\ \  \\
\mu _{2} &=&\frac{1+\alpha ^{3}+\left( 1-\alpha \lambda \right) ^{3}}{3}-%
\frac{\left( 1-\alpha \lambda \right) }{2}\left( 1+\alpha ^{2}\right) ,
\end{eqnarray*}%
\begin{eqnarray*}
\mu _{3} &=&\alpha \lambda \frac{\left( 1-\alpha \right) ^{2}}{2}-\frac{%
\left( 1-\alpha \right) ^{3}}{3},\  \\
\mu _{4} &=&\frac{\left( \alpha \lambda -1\right) \left( 1-\alpha
^{2}\right) }{2}+\frac{1-\alpha ^{3}}{3},
\end{eqnarray*}%
and%
\begin{eqnarray}
&&\dint\limits_{1-\alpha }^{1}\left\vert t-1+\lambda \left( 1-\alpha \right)
\right\vert \left\vert f^{\prime }\left( tb+(1-t)a\right) \right\vert ^{q}dt
\notag \\
&\leq &\dint\limits_{1-\alpha }^{1}\left\vert t-1+\lambda \left( 1-\alpha
\right) \right\vert \left[ t\left\vert f^{\prime }(b)\right\vert
^{q}+(1-t)\left\vert f^{\prime }(a)\right\vert ^{q}\right] dt  \notag \\
&=&\left\{ 
\begin{array}{cc}
\eta _{1}\left\vert f^{\prime }(b)\right\vert ^{q}+\eta _{2}\left\vert
f^{\prime }(a)\right\vert ^{q}, & 1-\lambda \left( 1-\alpha \right) \leq
1-\alpha \\ 
\eta _{3}\left\vert f^{\prime }(b)\right\vert ^{q}+\eta _{4}\left\vert
f^{\prime }(a)\right\vert ^{q}, & 1-\lambda \left( 1-\alpha \right) \geq
1-\alpha%
\end{array}%
\right. ,  \label{2-7}
\end{eqnarray}%
where $\eta _{1},\ \eta _{2},\ \eta _{3}$ and$\ \eta _{4}$ are defined as in
(\ref{2-2e}).$\ $Thus, using (\ref{2-4})-(\ref{2-7}) in (\ref{2-3}), we
obtain the inequality (\ref{2-2}). This completes the proof.
\end{proof}

\begin{corollary}
Let the assumptions of Theorem \ref{2.2} hold. Then for $q=1$ the inequality
(\ref{2-2}) reduced to the following inequality%
\begin{equation}
\left\vert \lambda \left( \alpha f(a)+\left( 1-\alpha \right) f(b)\right)
+\left( 1-\lambda \right) f(\alpha a+\left( 1-\alpha \right) b)-\frac{1}{b-a}%
\dint\limits_{a}^{b}f(x)dx\right\vert   \label{2-8}
\end{equation}%
\begin{equation*}
\leq \left\{ 
\begin{array}{cc}
\left( b-a\right) \left\{ \left( \mu _{1+}\eta _{3}\right) \left\vert
f^{\prime }(b)\right\vert +\left( \mu _{2}+\eta _{4}\right) \left\vert
f^{\prime }(a)\right\vert \right\} , & \alpha \lambda \leq 1-\alpha \leq
1-\lambda \left( 1-\alpha \right)  \\ 
\left( b-a\right) \left\{ \left( \mu _{1}+\eta _{1}\right) \left\vert
f^{\prime }(b)\right\vert +\left( \mu _{2}+\eta _{2}\right) \left\vert
f^{\prime }(a)\right\vert \right\} , & \alpha \lambda \leq 1-\lambda \left(
1-\alpha \right) \leq 1-\alpha  \\ 
\left( b-a\right) \left\{ \left( \mu _{3}+\eta _{3}\right) \left\vert
f^{\prime }(b)\right\vert +\left( \mu _{4}+\eta _{4}\right) \left\vert
f^{\prime }(a)\right\vert \right\} , & 1-\alpha \leq \alpha \lambda \leq
1-\lambda \left( 1-\alpha \right) 
\end{array}%
\right. 
\end{equation*}
\end{corollary}

\begin{corollary}
Let the assumptions of Theorem \ref{2.2} hold. Then for $\alpha =\frac{1}{2}$
and $\lambda =\frac{1}{3}$, from the inequality (\ref{2-2}) we get the
following Simpson type inequality 
\begin{eqnarray}
&&\left\vert \frac{1}{6}\left[ f(a)+4f\left( \frac{a+b}{2}\right) +f(b)%
\right] -\frac{1}{b-a}\dint\limits_{a}^{b}f(x)dx\right\vert  \label{2-9} \\
&\leq &\left( b-a\right) \left( \frac{5}{72}\right) ^{1-\frac{1}{q}}\left\{
\left( \frac{29}{1296}\left\vert f^{\prime }(b)\right\vert ^{q}+\frac{61}{%
1296}\left\vert f^{\prime }(a)\right\vert ^{q}\right) ^{\frac{1}{q}}\right. 
\notag \\
&&\left. +\left( \frac{61}{1296}\left\vert f^{\prime }(b)\right\vert ^{q}+%
\frac{29}{1296}\left\vert f^{\prime }(a)\right\vert ^{q}\right) ^{\frac{1}{q}%
}\right\} ,  \notag
\end{eqnarray}%
which is the same of the inequality in \cite[Theorem 10]{SSO10} for $s=1$ .
\end{corollary}

\begin{corollary}
\label{3}Let the assumptions of Theorem \ref{2.2} hold. Then for $\alpha =%
\frac{1}{2}$ and $\lambda =0,$from the inequality (\ref{2-2}) we get the
following midpoint type inequality%
\begin{eqnarray}
&&\left\vert f\left( \frac{a+b}{2}\right) -\frac{1}{b-a}\dint%
\limits_{a}^{b}f(x)dx\right\vert   \label{2-10} \\
&\leq &\frac{b-a}{8}\left\{ \left( \frac{\left\vert f^{\prime
}(b)\right\vert ^{q}+2\left\vert f^{\prime }(a)\right\vert ^{q}}{3}\right) ^{%
\frac{1}{q}}+\left( \frac{2\left\vert f^{\prime }(b)\right\vert
^{q}+\left\vert f^{\prime }(a)\right\vert ^{q}}{3}\right) ^{\frac{1}{q}%
}\right\}   \notag
\end{eqnarray}
\end{corollary}

\begin{corollary}
In Corollary \ref{3}, if $q=1$, then we have the following midpoint type
inequality%
\begin{equation}
\left\vert f\left( \frac{a+b}{2}\right) -\frac{1}{b-a}\dint%
\limits_{a}^{b}f(x)dx\right\vert \leq \frac{b-a}{4}\left( \frac{\left\vert
f^{\prime }(a)\right\vert +\left\vert f^{\prime }(b)\right\vert }{2}\right) .
\label{2-11}
\end{equation}%
which is the same of the inequality (\ref{1-3}).
\end{corollary}

\begin{corollary}
Let the assumptions of Theorem \ref{2.2} hold. Then for $\alpha =\frac{1}{2}$
and $\lambda =1,$from the inequality (\ref{2-2}) we get the following
trapezoid type inequality%
\begin{eqnarray*}
&&\left\vert \frac{f\left( a\right) +f\left( b\right) }{2}-\frac{1}{b-a}%
\dint\limits_{a}^{b}f(x)dx\right\vert  \\
&\leq &\frac{b-a}{8}\left\{ \left( \frac{\left\vert f^{\prime
}(b)\right\vert ^{q}+5\left\vert f^{\prime }(a)\right\vert ^{q}}{6}\right) ^{%
\frac{1}{q}}+\left( \frac{5\left\vert f^{\prime }(b)\right\vert
^{q}+\left\vert f^{\prime }(a)\right\vert ^{q}}{6}\right) ^{\frac{1}{q}%
}\right\} 
\end{eqnarray*}
\end{corollary}

Using Lemma \ref{2.1} we shall give another result for convex functions as
follows.

\begin{theorem}
\label{2.3}Let $f:I\subset \mathbb{R\rightarrow R}$ be a differentiable
mapping on $I^{\circ }$ such that $f^{\prime }\in L[a,b]$, where $a,b\in
I^{\circ }$ with $a<b$ and $\alpha ,\lambda \in \left[ 0,1\right] $. If $%
\left\vert f^{\prime }\right\vert ^{q}$ is convex on $[a,b]$, $q>1,$ then
the following inequality holds:%
\begin{equation}
\left\vert \lambda \left( \alpha f(a)+\left( 1-\alpha \right) f(b)\right)
+\left( 1-\lambda \right) f(\alpha a+\left( 1-\alpha \right) b)-\frac{1}{b-a}%
\dint\limits_{a}^{b}f(x)dx\right\vert \leq \left( b-a\right)  \label{2-12}
\end{equation}%
\begin{equation*}
\times \left( \frac{1}{p+1}\right) ^{\frac{1}{p}}\left\{ 
\begin{array}{cc}
\left[ \left( 1-\alpha \right) ^{\frac{1}{q}}\varepsilon _{1}^{\frac{1}{p}%
}\delta _{1}^{\frac{1}{q}}+\alpha ^{\frac{1}{q}}\varepsilon _{3}^{\frac{1}{p}%
}\delta _{2}^{\frac{1}{q}}\right] , & \alpha \lambda \leq 1-\alpha \leq
1-\lambda \left( 1-\alpha \right) \\ 
\left[ \left( 1-\alpha \right) ^{\frac{1}{q}}\varepsilon _{1}^{\frac{1}{p}%
}\delta _{1}^{\frac{1}{q}}+\alpha ^{\frac{1}{q}}\varepsilon _{4}^{\frac{1}{p}%
}\delta _{2}^{\frac{1}{q}}\right] , & \alpha \lambda \leq 1-\lambda \left(
1-\alpha \right) \leq 1-\alpha \\ 
\left[ \left( 1-\alpha \right) ^{\frac{1}{q}}\varepsilon _{2}^{\frac{1}{p}%
}\delta _{1}^{\frac{1}{q}}+\alpha ^{\frac{1}{q}}\varepsilon _{3}^{\frac{1}{p}%
}\delta _{2}^{\frac{1}{q}}\right] , & 1-\alpha \leq \alpha \lambda \leq
1-\lambda \left( 1-\alpha \right)%
\end{array}%
\right.
\end{equation*}%
where 
\begin{equation}
\delta _{1}=\frac{\left\vert f^{\prime }\left( \left( 1-\alpha \right)
b+\alpha a\right) \right\vert ^{q}+\left\vert f^{\prime }\left( a\right)
\right\vert ^{q}}{2},\ \delta _{2}=\frac{\left\vert f^{\prime }\left( \left(
1-\alpha \right) b+\alpha a\right) \right\vert ^{q}+\left\vert f^{\prime
}\left( b\right) \right\vert ^{q}}{2},  \label{2-12a}
\end{equation}%
\begin{eqnarray}
\varepsilon _{1} &=&\left( \alpha \lambda \right) ^{p+1}+\left( 1-\alpha
-\alpha \lambda \right) ^{p+1},\ \varepsilon _{2}=\left( \alpha \lambda
\right) ^{p+1}-\left( \alpha \lambda -1+\alpha \right) ^{p+1},  \notag \\
\varepsilon _{3} &=&\left[ \lambda \left( 1-\alpha \right) \right] ^{p+1}+%
\left[ \alpha -\lambda \left( 1-\alpha \right) \right] ^{p+1},\ \varepsilon
_{4}=\left[ \lambda \left( 1-\alpha \right) \right] ^{p+1}-\left[ \lambda
\left( 1-\alpha \right) -\alpha \right] ^{p+1},  \notag
\end{eqnarray}%
and $\frac{1}{p}+\frac{1}{q}=1.$
\end{theorem}

\begin{proof}
From Lemma \ref{2.1} and by H\"{o}lder's integral inequality, we have%
\begin{eqnarray*}
&&\left\vert \lambda \left( \alpha f(a)+\left( 1-\alpha \right) f(b)\right)
+\left( 1-\lambda \right) f(\alpha a+\left( 1-\alpha \right) b)-\frac{1}{b-a}%
\dint\limits_{a}^{b}f(x)dx\right\vert \\
&\leq &\left( b-a\right) \left[ \dint\limits_{0}^{1-\alpha }\left\vert
t-\alpha \lambda \right\vert \left\vert f^{\prime }\left( tb+(1-t)a\right)
\right\vert dt+\dint\limits_{1-\alpha }^{1}\left\vert t-1+\lambda \left(
1-\alpha \right) \right\vert \left\vert f^{\prime }\left( tb+(1-t)a\right)
\right\vert dt\right] \\
&\leq &\left( b-a\right) \left\{ \left( \dint\limits_{0}^{1-\alpha
}\left\vert t-\alpha \lambda \right\vert ^{p}dt\right) ^{\frac{1}{p}}\left(
\dint\limits_{0}^{1-\alpha }\left\vert f^{\prime }\left( tb+(1-t)a\right)
\right\vert ^{q}dt\right) ^{\frac{1}{q}}\right.
\end{eqnarray*}%
\begin{equation}
+\left. \left( \dint\limits_{1-\alpha }^{1}\left\vert t-1+\lambda \left(
1-\alpha \right) \right\vert ^{p}dt\right) ^{\frac{1}{p}}\left(
\dint\limits_{1-\alpha }^{1}\left\vert f^{\prime }\left( tb+(1-t)a\right)
\right\vert ^{q}dt\right) ^{\frac{1}{q}}\right\} .  \label{2-13}
\end{equation}%
Since $\left\vert f^{\prime }\right\vert ^{q}$ is convex on $[a,b],$ for $%
\alpha \in \left[ 0,1\right) $ by the inequality (\ref{1-1}), we get 
\begin{eqnarray}
\dint\limits_{0}^{1-\alpha }\left\vert f^{\prime }\left( tb+(1-t)a\right)
\right\vert ^{q}dt &=&\left( 1-\alpha \right) \left[ \frac{1}{\left(
1-\alpha \right) \left( b-a\right) }\dint\limits_{a}^{\left( 1-\alpha
\right) b+\alpha a}\left\vert f^{\prime }\left( x\right) \right\vert ^{q}dx%
\right]  \notag \\
&\leq &\left( 1-\alpha \right) \frac{\left\vert f^{\prime }\left( \left(
1-\alpha \right) b+\alpha a\right) \right\vert ^{q}+\left\vert f^{\prime
}\left( a\right) \right\vert ^{q}}{2}.  \label{2-14}
\end{eqnarray}%
The inequality (\ref{2-14}) holds for $\alpha =1$ too. Similarly, for $%
\alpha \in \left( 0,1\right] $ by the inequality (\ref{1-1}), we have 
\begin{eqnarray}
\dint\limits_{1-\alpha }^{1}\left\vert f^{\prime }\left( tb+(1-t)a\right)
\right\vert ^{q}dt &=&\alpha \left[ \frac{1}{\alpha \left( b-a\right) }%
\dint\limits_{\left( 1-\alpha \right) b+\alpha a}^{b}\left\vert f^{\prime
}\left( x\right) \right\vert ^{q}dx\right]  \notag \\
&\leq &\alpha \frac{\left\vert f^{\prime }\left( \left( 1-\alpha \right)
b+\alpha a\right) \right\vert ^{q}+\left\vert f^{\prime }\left( b\right)
\right\vert ^{q}}{2}.  \label{2-15}
\end{eqnarray}%
The inequality (\ref{2-15}) holds for $\alpha =0$ too. By simple computation%
\begin{equation}
\dint\limits_{0}^{1-\alpha }\left\vert t-\alpha \lambda \right\vert
^{p}dt=\left\{ 
\begin{array}{cc}
\frac{\left( \alpha \lambda \right) ^{p+1}+\left( 1-\alpha -\alpha \lambda
\right) ^{p+1}}{p+1}, & \alpha \lambda \leq 1-\alpha \\ 
\frac{\left( \alpha \lambda \right) ^{p+1}-\left( \alpha \lambda -1+\alpha
\right) ^{p+1}}{p+1}, & \alpha \lambda \geq 1-\alpha%
\end{array}%
\right. ,  \label{2-16}
\end{equation}%
and%
\begin{equation}
\dint\limits_{1-\alpha }^{1}\left\vert t-1+\lambda \left( 1-\alpha \right)
\right\vert ^{p}dt=\left\{ 
\begin{array}{cc}
\frac{\left[ \lambda \left( 1-\alpha \right) \right] ^{p+1}+\left[ \alpha
-\lambda \left( 1-\alpha \right) \right] ^{p+1}}{p+1}, & 1-\alpha \leq
1-\lambda \left( 1-\alpha \right) \\ 
\frac{\left[ \lambda \left( 1-\alpha \right) \right] ^{p+1}-\left[ \lambda
\left( 1-\alpha \right) -\alpha \right] ^{p+1}}{p+1}, & 1-\alpha \geq
1-\lambda \left( 1-\alpha \right)%
\end{array}%
\right. ,  \label{2-17}
\end{equation}%
thus, using (\ref{2-14})-(\ref{2-17}) in (\ref{2-13}), we obtain the
inequality (\ref{2-12}). This completes the proof.
\end{proof}

\begin{corollary}
Let the assumptions of Theorem \ref{2.3} hold. Then for $\alpha =\frac{1}{2}$
and $\lambda =\frac{1}{3}$, from the inequality (\ref{2-12}) we get the
following Simpson type inequality 
\begin{equation}
\left\vert \frac{1}{6}\left[ f(a)+4f\left( \frac{a+b}{2}\right) +f(b)\right]
-\frac{1}{b-a}\dint\limits_{a}^{b}f(x)dx\right\vert  \label{2-18}
\end{equation}%
\begin{equation*}
\leq \frac{b-a}{12}\left( \frac{1+2^{p+1}}{3\left( p+1\right) }\right) ^{%
\frac{1}{p}}\left\{ \left( \frac{\left\vert f^{\prime }\left( \frac{a+b}{2}%
\right) \right\vert ^{q}+\left\vert f^{\prime }\left( a\right) \right\vert
^{q}}{2}\right) ^{\frac{1}{q}}+\left( \frac{\left\vert f^{\prime }\left( 
\frac{a+b}{2}\right) \right\vert ^{q}+\left\vert f^{\prime }\left( b\right)
\right\vert ^{q}}{2}\right) ^{\frac{1}{q}}\right\} ,
\end{equation*}%
which is the same of the inequality (\ref{1-2}).
\end{corollary}

\begin{remark}
We note that if we use convexity of $\left\vert f^{\prime }\right\vert ^{q}$
in the inequality (\ref{2-18}) then we obtain the inequality (\ref{1-2a}).
\end{remark}

\begin{corollary}
Let the assumptions of Theorem \ref{2.3} hold. Then for $\alpha =\frac{1}{2}$
and $\lambda =0,$ from the inequality (\ref{2-12}) we get the following
midpoint type inequality%
\begin{eqnarray*}
&&\left\vert f\left( \frac{a+b}{2}\right) -\frac{1}{b-a}\dint%
\limits_{a}^{b}f(x)dx\right\vert  \\
&\leq &\frac{b-a}{4}\left( \frac{1}{p+1}\right) ^{\frac{1}{p}}\left\{ \left( 
\frac{\left\vert f^{\prime }\left( \frac{a+b}{2}\right) \right\vert
^{q}+\left\vert f^{\prime }\left( a\right) \right\vert ^{q}}{2}\right) ^{%
\frac{1}{q}}+\left( \frac{\left\vert f^{\prime }\left( \frac{a+b}{2}\right)
\right\vert ^{q}+\left\vert f^{\prime }\left( b\right) \right\vert ^{q}}{2}%
\right) ^{\frac{1}{q}}\right\} .
\end{eqnarray*}
\end{corollary}

\begin{corollary}
Let the assumptions of Theorem \ref{2.3} hold. Then for $\alpha =\frac{1}{2}$
and $\lambda =1,$ from the inequality (\ref{2-12}) we get the following
trapezoid type inequality%
\begin{eqnarray*}
&&\left\vert \frac{f\left( a\right) +f\left( b\right) }{2}-\frac{1}{b-a}%
\dint\limits_{a}^{b}f(x)dx\right\vert  \\
&\leq &\frac{b-a}{4}\left( \frac{1}{p+1}\right) ^{\frac{1}{p}}\left\{ \left( 
\frac{\left\vert f^{\prime }\left( \frac{a+b}{2}\right) \right\vert
^{q}+\left\vert f^{\prime }\left( a\right) \right\vert ^{q}}{2}\right) ^{%
\frac{1}{q}}+\left( \frac{\left\vert f^{\prime }\left( \frac{a+b}{2}\right)
\right\vert ^{q}+\left\vert f^{\prime }\left( b\right) \right\vert ^{q}}{2}%
\right) ^{\frac{1}{q}}\right\} .
\end{eqnarray*}
\end{corollary}

\begin{theorem}
\label{2.4}Let $f:I\subset \mathbb{R\rightarrow R}$ be a differentiable
mapping on $I^{\circ }$ such that $f^{\prime }\in L[a,b]$, where $a,b\in
I^{\circ }$ with $a<b$ and $\alpha ,\lambda \in \left[ 0,1\right] $. If $%
\left\vert f^{\prime }\right\vert ^{q}$ is convex on $[a,b]$, $q>1,$ then
the following inequality holds,%
\begin{equation}
\left\vert \lambda \left( \alpha f(a)+\left( 1-\alpha \right) f(b)\right)
+\left( 1-\lambda \right) f(\alpha a+\left( 1-\alpha \right) b)-\frac{1}{b-a}%
\dint\limits_{a}^{b}f(x)dx\right\vert \leq \left( b-a\right)  \label{2-18a}
\end{equation}%
\begin{equation*}
\times \left( \frac{1}{p+1}\right) ^{\frac{1}{p}}\left\{ 
\begin{array}{cc}
\left[ \varepsilon _{1}^{\frac{1}{p}}\delta _{3}^{\frac{1}{q}}+\varepsilon
_{3}^{\frac{1}{p}}\delta _{4}^{\frac{1}{q}}\right] , & \alpha \lambda \leq
1-\alpha \leq 1-\lambda \left( 1-\alpha \right) \\ 
\left[ \varepsilon _{1}^{\frac{1}{p}}\delta _{3}^{\frac{1}{q}}+\varepsilon
_{4}^{\frac{1}{p}}\delta _{4}^{\frac{1}{q}}\right] , & \alpha \lambda \leq
1-\lambda \left( 1-\alpha \right) \leq 1-\alpha \\ 
\left[ \varepsilon _{2}^{\frac{1}{p}}\delta _{3}^{\frac{1}{q}}+\varepsilon
_{3}^{\frac{1}{p}}\delta _{4}^{\frac{1}{q}}\right] , & 1-\alpha \leq \alpha
\lambda \leq 1-\lambda \left( 1-\alpha \right)%
\end{array}%
\right.
\end{equation*}%
where 
\begin{eqnarray*}
\delta _{3} &=&\frac{\left\vert f^{\prime }\left( b\right) \right\vert
^{q}\left( 1-\alpha \right) ^{2}+(1-\alpha ^{2})\left\vert f^{\prime }\left(
a\right) \right\vert ^{q}dt}{2},\ \delta _{4}=\frac{\left\vert f^{\prime
}\left( b\right) \right\vert ^{q}\alpha \left( 2-\alpha \right) +\alpha
^{2}\left\vert f^{\prime }\left( a\right) \right\vert ^{q}dt}{2}, \\
\varepsilon _{1} &=&\left( \alpha \lambda \right) ^{p+1}+\left( 1-\alpha
-\alpha \lambda \right) ^{p+1},\ \varepsilon _{2}=\left( \alpha \lambda
\right) ^{p+1}-\left( \alpha \lambda -1+\alpha \right) ^{p+1}, \\
\varepsilon _{3} &=&\left[ \lambda \left( 1-\alpha \right) \right] ^{p+1}+%
\left[ \alpha -\lambda \left( 1-\alpha \right) \right] ^{p+1},\ \varepsilon
_{4}=\left[ \lambda \left( 1-\alpha \right) \right] ^{p+1}-\left[ \lambda
\left( 1-\alpha \right) -\alpha \right] ^{p+1},
\end{eqnarray*}%
and $\frac{1}{p}+\frac{1}{q}=1.$
\end{theorem}

\begin{proof}
From Lemma \ref{2.1} and by H\"{o}lder's integral inequality, we have the
inequality (\ref{2-13}). Since Since $\left\vert f^{\prime }\right\vert ^{q}$
is convex on $[a,b],$ we know that for $t\in \left[ 0,1-\alpha \right] $ and 
$t\in \left[ 1-\alpha ,1\right] $ 
\begin{equation*}
\left\vert f^{\prime }\left( tb+(1-t)a\right) \right\vert ^{q}\leq
t\left\vert f^{\prime }\left( b\right) \right\vert ^{q}+(1-t)\left\vert
f^{\prime }\left( a\right) \right\vert ^{q}.
\end{equation*}%
Hence 
\begin{eqnarray*}
&&\left\vert \lambda \left( \alpha f(a)+\left( 1-\alpha \right) f(b)\right)
+\left( 1-\lambda \right) f(\alpha a+\left( 1-\alpha \right) b)-\frac{1}{b-a}%
\dint\limits_{a}^{b}f(x)dx\right\vert \\
&\leq &\left( b-a\right) \left\{ \left( \dint\limits_{0}^{1-\alpha
}\left\vert t-\alpha \lambda \right\vert ^{p}dt\right) ^{\frac{1}{p}}\left(
\dint\limits_{0}^{1-\alpha }t\left\vert f^{\prime }\left( b\right)
\right\vert ^{q}+(1-t)\left\vert f^{\prime }\left( a\right) \right\vert
^{q}dt\right) ^{\frac{1}{q}}\right. \\
&&+\left. \left( \dint\limits_{1-\alpha }^{1}\left\vert t-1+\lambda \left(
1-\alpha \right) \right\vert ^{p}dt\right) ^{\frac{1}{p}}\left(
\dint\limits_{1-\alpha }^{1}t\left\vert f^{\prime }\left( b\right)
\right\vert ^{q}+(1-t)\left\vert f^{\prime }\left( a\right) \right\vert
^{q}dt\right) ^{\frac{1}{q}}\right\} \\
&\leq &\left( b-a\right) \left\{ \left( \dint\limits_{0}^{1-\alpha
}\left\vert t-\alpha \lambda \right\vert ^{p}dt\right) ^{\frac{1}{p}}\left( 
\frac{\left\vert f^{\prime }\left( b\right) \right\vert ^{q}\left( 1-\alpha
\right) ^{2}+(1-\alpha ^{2})\left\vert f^{\prime }\left( a\right)
\right\vert ^{q}dt}{2}\right) ^{\frac{1}{q}}\right.
\end{eqnarray*}%
\begin{equation}
+\left. \left( \dint\limits_{1-\alpha }^{1}\left\vert t-1+\lambda \left(
1-\alpha \right) \right\vert ^{p}dt\right) ^{\frac{1}{p}}\left( \frac{%
\left\vert f^{\prime }\left( b\right) \right\vert ^{q}\alpha \left( 2-\alpha
\right) +\alpha ^{2}\left\vert f^{\prime }\left( a\right) \right\vert ^{q}dt%
}{2}\right) ^{\frac{1}{q}}\right\} .  \label{2-19}
\end{equation}%
thus, using (\ref{2-16}),(\ref{2-17}) in (\ref{2-19}), we obtain the
inequality (\ref{2-18a}). This completes the proof.
\end{proof}

\begin{corollary}
Let the assumptions of Theorem \ref{2.4} hold. Then for $\alpha =\frac{1}{2}$
and $\lambda =\frac{1}{3}$, from the inequality (\ref{2-12}) we get the
following Simpson type inequality%
\begin{eqnarray*}
&&\left\vert \frac{1}{6}\left[ f(a)+4f\left( \frac{a+b}{2}\right) +f(b)%
\right] -\frac{1}{b-a}\dint\limits_{a}^{b}f(x)dx\right\vert \\
&\leq &\frac{b-a}{12}\left( \frac{1+2^{p+1}}{3\left( p+1\right) }\right) ^{%
\frac{1}{p}}\left\{ \left( \frac{3\left\vert f^{\prime }(b)\right\vert
^{q}+\left\vert f^{\prime }\left( a\right) \right\vert ^{q}}{4}\right) ^{%
\frac{1}{q}}+\left( \frac{3\left\vert f^{\prime }(a)\right\vert
^{q}+\left\vert f^{\prime }\left( b\right) \right\vert ^{q}}{4}\right) ^{%
\frac{1}{q}}\right\} ,
\end{eqnarray*}%
which is the same of the inequality (\ref{1-2a}).
\end{corollary}

\section{Some applications for special means}

Let us recall the following special means of arbitrary real numbers $a,b$
with $a\neq b$ and $\alpha \in \left[ 0,1\right] :$

\begin{enumerate}
\item The weighted arithmetic mean%
\begin{equation*}
A_{\alpha }\left( a,b\right) :=\alpha a+(1-\alpha )b,~a,b\in 
\mathbb{R}
.
\end{equation*}

\item The unweighted arithmetic mean%
\begin{equation*}
A\left( a,b\right) :=\frac{a+b}{2},~a,b\in 
\mathbb{R}
.
\end{equation*}

\item The weighted geometric mean%
\begin{equation*}
G_{\alpha }(a,b)=a^{\alpha }b^{1-\alpha },\ a,b>0.
\end{equation*}

\item The unweighted geometric mean%
\begin{equation*}
G(a,b)=\sqrt{ab},\ a,b>0.
\end{equation*}

\item The weighted harmonic mean%
\begin{equation*}
H_{\alpha }\left( a,b\right) :=\left( \frac{\alpha }{a}+\frac{1-\alpha }{b}%
\right) ^{-1},\ \ a,b\in 
\mathbb{R}
\backslash \left\{ 0\right\} .
\end{equation*}

\item The unweighted harmonic mean%
\begin{equation*}
H\left( a,b\right) :=\frac{2ab}{a+b},\ \ a,b\in 
\mathbb{R}
\backslash \left\{ 0\right\} .
\end{equation*}

\item The Logarithmic mean%
\begin{equation*}
L\left( a,b\right) :=\frac{b-a}{\ln \left\vert b\right\vert -\ln \left\vert
a\right\vert },\ \ \left\vert a\right\vert \neq \left\vert b\right\vert ,\
ab\neq 0.
\end{equation*}

\item Then n-Logarithmic mean%
\begin{equation*}
L_{n}\left( a,b\right) :=\ \left( \frac{b^{n+1}-a^{n+1}}{(n+1)(b-a)}\right)
^{\frac{1}{n}}\ ,\ n\in 
\mathbb{Z}
\backslash \left\{ -1,0\right\} ,\ a,b\in 
\mathbb{R}
,\ a\neq b.
\end{equation*}

\item The identric mean%
\begin{equation*}
I(a,b)=\frac{1}{e}\left( \frac{b^{b}}{a^{a}}\right) ^{\frac{1}{b-a}},\
a,b>0,\ a\neq b.
\end{equation*}
\end{enumerate}

\begin{proposition}
Let $a,b\in 
\mathbb{R}
$ with $a<b,\ 0\notin \left[ a,b\right] $ and $n\in 
\mathbb{Z}
,\ \left\vert n\right\vert \geq 2.$ Then, for $\alpha ,\lambda \in \left[ 0,1%
\right] $ and $q\geq 1,$we have the following inequality:%
\begin{eqnarray*}
&&\left\vert \lambda A_{\alpha }\left( a^{n},b^{n}\right) +\left( 1-\lambda
\right) A_{\alpha }^{n}\left( a,b\right) -L_{n}^{n}\left( a,b\right)
\right\vert \\
&\leq &\left\{ 
\begin{array}{cc}
\begin{array}{c}
\left( b-a\right) \left\vert n\right\vert \left\{ \gamma _{2}^{1-\frac{1}{q}%
}\left( \mu _{1}\left\vert b\right\vert ^{(n-1)q}+\mu _{2}\left\vert
a\right\vert ^{(n-1)q}\right) ^{\frac{1}{q}}\right. \\ 
\left. +\upsilon _{2}^{1-\frac{1}{q}}\left( \eta _{3}\left\vert b\right\vert
^{(n-1)q}+\eta _{4}\left\vert a\right\vert ^{(n-1)q}\right) ^{\frac{1}{q}%
}\right\} ,%
\end{array}
& \alpha \lambda \leq 1-\alpha \leq 1-\lambda \left( 1-\alpha \right) \\ 
\begin{array}{c}
\left( b-a\right) \left\vert n\right\vert \left\{ \gamma _{2}^{1-\frac{1}{q}%
}\left( \mu _{1}\left\vert b\right\vert ^{(n-1)q}+\mu _{2}\left\vert
a\right\vert ^{(n-1)q}\right) ^{\frac{1}{q}}\right. \\ 
\left. +\upsilon _{1}^{1-\frac{1}{q}}\left( \eta _{1}\left\vert b\right\vert
^{(n-1)q}+\eta _{2}\left\vert a\right\vert ^{(n-1)q}\right) ^{\frac{1}{q}%
}\right\} ,%
\end{array}
& \alpha \lambda \leq 1-\lambda \left( 1-\alpha \right) \leq 1-\alpha \\ 
\begin{array}{c}
\left( b-a\right) \left\vert n\right\vert \left\{ \gamma _{1}^{1-\frac{1}{q}%
}\left( \mu _{3}\left\vert b\right\vert ^{(n-1)q}+\mu _{4}\left\vert
a\right\vert ^{(n-1)q}\right) ^{\frac{1}{q}}\right. \\ 
\left. +\upsilon _{2}^{1-\frac{1}{q}}\left( \eta _{3}\left\vert b\right\vert
^{(n-1)q}+\eta _{4}\left\vert a\right\vert ^{(n-1)q}\right) ^{\frac{1}{q}%
}\right\} ,%
\end{array}
& 1-\alpha \leq \alpha \lambda \leq 1-\lambda \left( 1-\alpha \right)%
\end{array}%
\right. ,
\end{eqnarray*}%
where $\gamma _{1},\ \gamma _{2},\ \upsilon _{1},\ \ \upsilon _{2},$ $\mu
_{1},\ \mu _{2},\ \mu _{3},\ \mu _{4},\ \eta _{1},\ \eta _{2},\ \eta _{3},\
\eta _{4}$ numbers are defined as in (\ref{2-2a})-(\ref{2-2e}).
\end{proposition}

\begin{proof}
The assertion follows from Theorem \ref{2.2}, for$\ f(x)=x^{n},\ x\in 
\mathbb{R}
,\ n\in 
\mathbb{Z}
,\ \left\vert n\right\vert \geq 2.$
\end{proof}

\begin{proposition}
Let $a,b\in 
\mathbb{R}
$ with $a<b,\ 0\notin \left[ a,b\right] ,$ and $n\in 
\mathbb{Z}
,\ \left\vert n\right\vert \geq 2.$ Then, for $\alpha ,\lambda \in \left[ 0,1%
\right] $ and $q>1,$ we have the following inequality:%
\begin{equation*}
\left\vert \lambda A_{\alpha }\left( a^{n},b^{n}\right) +\left( 1-\lambda
\right) A_{\alpha }^{n}\left( a,b\right) -L_{n}^{n}\left( a,b\right)
\right\vert \leq \left( b-a\right) \left( \frac{1}{p+1}\right) ^{\frac{1}{p}%
}\left\vert n\right\vert
\end{equation*}%
\begin{equation*}
\times \left\{ 
\begin{array}{cc}
\left[ \left( 1-\alpha \right) ^{\frac{1}{q}}\varepsilon _{1}^{\frac{1}{p}%
}\theta _{1}+\alpha ^{\frac{1}{q}}\varepsilon _{3}^{\frac{1}{p}}\theta _{2}%
\right] , & \alpha \lambda \leq 1-\alpha \leq 1-\lambda \left( 1-\alpha
\right) \\ 
\left[ \left( 1-\alpha \right) ^{\frac{1}{q}}\varepsilon _{1}^{\frac{1}{p}%
}\theta _{1}+\alpha ^{\frac{1}{q}}\varepsilon _{4}^{\frac{1}{p}}\theta _{2}%
\right] , & \alpha \lambda \leq 1-\lambda \left( 1-\alpha \right) \leq
1-\alpha \\ 
\left[ \left( 1-\alpha \right) ^{\frac{1}{q}}\varepsilon _{2}^{\frac{1}{p}%
}\theta _{1}+\alpha ^{\frac{1}{q}}\varepsilon _{3}^{\frac{1}{p}}\theta _{2}%
\right] , & 1-\alpha \leq \alpha \lambda \leq 1-\lambda \left( 1-\alpha
\right)%
\end{array}%
\right. ,
\end{equation*}%
where 
\begin{equation*}
\theta _{1}=A^{\frac{1}{q}}\left( \left\vert A_{\alpha }^{(n-1)q}\left(
a,b\right) \right\vert ,\left\vert a\right\vert ^{(n-1)q}\right) ,\ \theta
_{2}=A^{\frac{1}{q}}\left( \left\vert A_{\alpha }^{(n-1)q}\left( a,b\right)
\right\vert ,\left\vert b\right\vert ^{(n-1)q}\right) ,\ \frac{1}{p}+\frac{1%
}{q}=1,
\end{equation*}%
and $\varepsilon _{1},\ \varepsilon _{2},\ \varepsilon _{3},\ \varepsilon
_{4}$ numbers are defined as in (\ref{2-12a}).
\end{proposition}

\begin{proof}
The assertion follows from Theorem \ref{2.3}, for$\ f(x)=x^{n},\ x\in 
\mathbb{R}
,\ n\in 
\mathbb{Z}
,\ \left\vert n\right\vert \geq 2.$
\end{proof}

\begin{proposition}
Let $a,b\in 
\mathbb{R}
$ with $a<b,\ 0\notin \left[ a,b\right] .$ Then, for $\alpha ,\lambda \in %
\left[ 0,1\right] $ and $q\geq 1,$ we have the following inequality:%
\begin{eqnarray*}
&&\left\vert \lambda H_{\alpha }^{-1}\left( a,b\right) +\left( 1-\lambda
\right) A_{\alpha }^{-1}\left( a,b\right) -L^{-1}\left( a,b\right)
\right\vert \\
&\leq &\left\{ 
\begin{array}{cc}
\begin{array}{c}
\left( b-a\right) \left\{ \gamma _{2}^{1-\frac{1}{q}}\left( \mu _{1}\frac{1}{%
\left\vert b\right\vert ^{2q}}+\mu _{2}\frac{1}{\left\vert a\right\vert ^{2q}%
}\right) ^{\frac{1}{q}}\right. \\ 
\left. +\upsilon _{2}^{1-\frac{1}{q}}\left( \eta _{3}\frac{1}{\left\vert
b\right\vert ^{2q}}+\eta _{4}\frac{1}{\left\vert a\right\vert ^{2q}}\right)
^{\frac{1}{q}}\right\} ,%
\end{array}
& \alpha \lambda \leq 1-\alpha \leq 1-\lambda \left( 1-\alpha \right) \\ 
\begin{array}{c}
\left( b-a\right) \left\{ \gamma _{2}^{1-\frac{1}{q}}\left( \mu _{1}\frac{1}{%
\left\vert b\right\vert ^{2q}}+\mu _{2}\frac{1}{\left\vert a\right\vert ^{2q}%
}\right) ^{\frac{1}{q}}\right. \\ 
\left. +\upsilon _{1}^{1-\frac{1}{q}}\left( \eta _{1}\frac{1}{\left\vert
b\right\vert ^{2q}}+\eta _{2}\frac{1}{\left\vert a\right\vert ^{2q}}\right)
^{\frac{1}{q}}\right\} ,%
\end{array}
& \alpha \lambda \leq 1-\lambda \left( 1-\alpha \right) \leq 1-\alpha \\ 
\begin{array}{c}
\left( b-a\right) \left\{ \gamma _{1}^{1-\frac{1}{q}}\left( \mu _{3}\frac{1}{%
\left\vert b\right\vert ^{2q}}+\mu _{4}\frac{1}{\left\vert a\right\vert ^{2q}%
}\right) ^{\frac{1}{q}}\right. \\ 
\left. +\upsilon _{2}^{1-\frac{1}{q}}\left( \eta _{3}\frac{1}{\left\vert
b\right\vert ^{2q}}+\eta _{4}\frac{1}{\left\vert a\right\vert ^{2q}}\right)
^{\frac{1}{q}}\right\} ,%
\end{array}
& 1-\alpha \leq \alpha \lambda \leq 1-\lambda \left( 1-\alpha \right)%
\end{array}%
\right. ,
\end{eqnarray*}%
where $\gamma _{1},\ \gamma _{2},\ \upsilon _{1},\ \ \upsilon _{2},$ $\mu
_{1},\ \mu _{2},\ \mu _{3},\ \mu _{4},\ \eta _{1},\ \eta _{2},\ \eta _{3},\
\eta _{4}$ numbers are defined as in (\ref{2-2a})-(\ref{2-2e}).
\end{proposition}

\begin{proof}
The assertion follows from Theorem \ref{2.2}, for$\ f(x)=\frac{1}{x},\ x\in 
\mathbb{R}
\backslash \left\{ 0\right\} .$
\end{proof}

\begin{proposition}
Let $a,b\in 
\mathbb{R}
$ with $0<a<b.$ Then, for $\alpha ,\lambda \in \left[ 0,1\right] $ and $q>1,$
we have the following inequality:%
\begin{equation*}
\left\vert \lambda H_{\alpha }^{-1}\left( a,b\right) +\left( 1-\lambda
\right) A_{\alpha }^{-1}\left( a,b\right) -L^{-1}\left( a,b\right)
\right\vert \leq \left( b-a\right) \left( \frac{1}{p+1}\right) ^{\frac{1}{p}}
\end{equation*}%
\begin{equation*}
\times \left\{ 
\begin{array}{cc}
\left[ \left( 1-\alpha \right) ^{\frac{1}{q}}\varepsilon _{1}^{\frac{1}{p}%
}\theta _{3}+\alpha ^{\frac{1}{q}}\varepsilon _{3}^{\frac{1}{p}}\theta _{4}%
\right] , & \alpha \lambda \leq 1-\alpha \leq 1-\lambda \left( 1-\alpha
\right)  \\ 
\left[ \left( 1-\alpha \right) ^{\frac{1}{q}}\varepsilon _{1}^{\frac{1}{p}%
}\theta _{3}+\alpha ^{\frac{1}{q}}\varepsilon _{4}^{\frac{1}{p}}\theta _{4}%
\right] , & \alpha \lambda \leq 1-\lambda \left( 1-\alpha \right) \leq
1-\alpha  \\ 
\left[ \left( 1-\alpha \right) ^{\frac{1}{q}}\varepsilon _{2}^{\frac{1}{p}%
}\theta _{3}+\alpha ^{\frac{1}{q}}\varepsilon _{3}^{\frac{1}{p}}\theta _{4}%
\right] , & 1-\alpha \leq \alpha \lambda \leq 1-\lambda \left( 1-\alpha
\right) 
\end{array}%
\right. ,
\end{equation*}%
where 
\begin{equation*}
\theta _{3}=H^{-\frac{1}{q}}\left( A_{\alpha }^{2q}\left( a,b\right)
,a^{2q}\right) ,\ \theta _{4}=H^{-\frac{1}{q}}\left( A_{\alpha }^{2q}\left(
a,b\right) ,b^{2q}\right) ,\ \frac{1}{p}+\frac{1}{q}=1,
\end{equation*}%
and $\varepsilon _{1},\ \varepsilon _{2},\ \varepsilon _{3},\ \varepsilon
_{4}$ numbers are defined as in (\ref{2-12a}).
\end{proposition}

\begin{proof}
The assertion follows from Theorem \ref{2.3}, for$\ f(x)=\frac{1}{x}$,$\
x\in 
\mathbb{R}
\backslash \left\{ 0\right\} .$
\end{proof}

\begin{proposition}
Let $a,b\in 
\mathbb{R}
$ with $0<a<b.$ Then, for $\alpha ,\lambda \in \left[ 0,1\right] $ and $%
q\geq 1,$ we have the following inequality:%
\begin{eqnarray*}
&&\left\vert A_{\lambda }\left( \ln G_{\alpha }(a,b),\ln A_{\alpha
}(a,b)\right) -\ln I(a,b)\right\vert \\
&\leq &\left\{ 
\begin{array}{cc}
\begin{array}{c}
\left( b-a\right) \left\{ \gamma _{2}^{1-\frac{1}{q}}\left( \mu _{1}\frac{1}{%
b^{q}}+\mu _{2}\frac{1}{a^{q}}\right) ^{\frac{1}{q}}\right. \\ 
\left. +\upsilon _{2}^{1-\frac{1}{q}}\left( \eta _{3}\frac{1}{b^{q}}+\eta
_{4}\frac{1}{a^{q}}\right) ^{\frac{1}{q}}\right\} ,%
\end{array}
& \alpha \lambda \leq 1-\alpha \leq 1-\lambda \left( 1-\alpha \right) \\ 
\begin{array}{c}
\left( b-a\right) \left\{ \gamma _{2}^{1-\frac{1}{q}}\left( \mu _{1}\frac{1}{%
b^{q}}+\mu _{2}\frac{1}{a^{q}}\right) ^{\frac{1}{q}}\right. \\ 
\left. +\upsilon _{1}^{1-\frac{1}{q}}\left( \eta _{1}\frac{1}{b^{q}}+\eta
_{2}\frac{1}{a^{q}}\right) ^{\frac{1}{q}}\right\} ,%
\end{array}
& \alpha \lambda \leq 1-\lambda \left( 1-\alpha \right) \leq 1-\alpha \\ 
\begin{array}{c}
\left( b-a\right) \left\{ \gamma _{1}^{1-\frac{1}{q}}\left( \mu _{3}\frac{1}{%
b^{q}}+\mu _{4}\frac{1}{a^{q}}\right) ^{\frac{1}{q}}\right. \\ 
\left. +\upsilon _{2}^{1-\frac{1}{q}}\left( \eta _{3}\frac{1}{b^{q}}+\eta
_{4}\frac{1}{a^{q}}\right) ^{\frac{1}{q}}\right\} ,%
\end{array}
& 1-\alpha \leq \alpha \lambda \leq 1-\lambda \left( 1-\alpha \right)%
\end{array}%
\right.
\end{eqnarray*}%
where $\gamma _{1},\ \gamma _{2},\ \upsilon _{1},\ \ \upsilon _{2},$ $\mu
_{1},\ \mu _{2},\ \mu _{3},\ \mu _{4},\ \eta _{1},\ \eta _{2},\ \eta _{3},\
\eta _{4}$ numbers are defined as in (\ref{2-2a})-(\ref{2-2e}).
\end{proposition}

\begin{proof}
The assertion follows from Theorem \ref{2.2}, for$\ f(x)=-\ln x,\ x>0.$
\end{proof}

\begin{proposition}
Let $a,b\in 
\mathbb{R}
$ with $0<a<b.$ Then, for $\alpha ,\lambda \in \left[ 0,1\right] $ and $q>1,$
we have the following inequality:%
\begin{equation*}
\left\vert A_{\lambda }\left( \ln G_{\alpha }(a,b),\ln A_{\alpha
}(a,b)\right) -\ln I(a,b)\right\vert \leq \left( b-a\right) \left( \frac{1}{%
p+1}\right) ^{\frac{1}{p}}
\end{equation*}%
\begin{equation*}
\times \left\{ 
\begin{array}{cc}
\left[ \left( 1-\alpha \right) ^{\frac{1}{q}}\varepsilon _{1}^{\frac{1}{p}%
}\theta _{3}+\alpha ^{\frac{1}{q}}\varepsilon _{3}^{\frac{1}{p}}\theta _{4}%
\right] , & \alpha \lambda \leq 1-\alpha \leq 1-\lambda \left( 1-\alpha
\right) \\ 
\left[ \left( 1-\alpha \right) ^{\frac{1}{q}}\varepsilon _{1}^{\frac{1}{p}%
}\theta _{3}+\alpha ^{\frac{1}{q}}\varepsilon _{4}^{\frac{1}{p}}\theta _{4}%
\right] , & \alpha \lambda \leq 1-\lambda \left( 1-\alpha \right) \leq
1-\alpha \\ 
\left[ \left( 1-\alpha \right) ^{\frac{1}{q}}\varepsilon _{2}^{\frac{1}{p}%
}\theta _{3}+\alpha ^{\frac{1}{q}}\varepsilon _{3}^{\frac{1}{p}}\theta _{4}%
\right] , & 1-\alpha \leq \alpha \lambda \leq 1-\lambda \left( 1-\alpha
\right)%
\end{array}%
\right. ,
\end{equation*}%
where 
\begin{equation*}
\theta _{3}=H^{-\frac{1}{q}}\left( A_{\alpha }^{q}\left( a,b\right)
,a^{q}\right) ,\ \theta _{4}=H^{-\frac{1}{q}}\left( A_{\alpha }^{q}\left(
a,b\right) ,b^{q}\right) ,\ \frac{1}{p}+\frac{1}{q}=1,
\end{equation*}%
and $\varepsilon _{1},\ \varepsilon _{2},\ \varepsilon _{3},\ \varepsilon
_{4}$ numbers are defined as in (\ref{2-12a}).
\end{proposition}

\begin{proof}
The assertion follows from Theorem \ref{2.3}, for$\ f(x)=-\ln x,\ x>0.$
\end{proof}


\begin{thebibliography}{9}
\bibitem{ADK11} M.W. Alomari, M.Darus, U.S. Kirmaci, Some inequalities of
Hermite-Hadamard type for $s-$convex functions, Acta Math. Scientia, 31B (4)
(2011)1643-1652.

\bibitem{DP00} S.S. Dragomir and C.E.M. Pearce, Selected Topics on
Hermite-Hadamard Inequalities and Applications, RGMIA Monographs, Victoria
University, 2000.

\bibitem{K04} U.S. Kirmaci, Inequalities for differentiable mappings and
applications to special means of real numbers and to midpoint formula, Appl.
Math. Comput. 147 (2004) 137-146.

\bibitem{THH12} K.L. Tseng, S.R. Hwang, K.C. Hsu, Hadamard-type and
Bullen-type inequalities for Lipschitzian functions and their applications,
Computers and Mathematics with Applications, in press.

\bibitem{ADD09} M. Alomari, M. Darus, S.S. Dragomir, New inequalities of
Simpson's Type for $s-$convex functions with applications, RGMIA Res. Rep.
Coll. 12 (4) (2009) Article 9. Online http://ajmaa.org/RGMIA/v12n4.php.

\bibitem{SA11} M.Z. Sarikaya, N. Aktan, On the generalization of some
integral inequalities and their applications, Mathematical and Computer
Modelling, 54 (2011) 2175-2182.

\bibitem{SSO10} M.Z. Sarikaya, E. Set, M.E. \"{O}zdemir, On new inequalities
of Simpson's type for $s-$convex functions, Computers and Mathematics with
Applications 60 (2010) 2191-2199.

\bibitem{SSO10a} M.Z. Sarikaya, E. Set, M.E. \"{O}zdemir, On new
inequalities of Simpson's type for convex functions, RGMIA Res. Rep. Coll.
13 (2) (2010) Article 2.
\end{thebibliography}
\end{document}